\documentclass[12pt,twoside,leqno]{article}
\usepackage{amsmath}
\usepackage{amssymb}
\usepackage{amsxtra}
\usepackage{amscd}
\usepackage{amsthm}
\usepackage[mathscr]{eucal}
\usepackage[all]{xy}
\theoremstyle{plain}
\newtheorem{thm}[subsection]{Theorem}
\newtheorem{prop}[subsection]{Proposition}
\newtheorem{cor}[subsection]{Corollary}
\newtheorem{lem}[subsection]{Lemma}

\newtheorem{sbprop}[subsubsection]{Proposition}

\theoremstyle{definition}
\newtheorem{defn}[subsection]{Definition}

\newtheorem{rem}[subsection]{Remark}
\newtheorem{para}[subsection]{}

\newtheorem{sbrem}[subsubsection]{Remark}

\newenvironment{pf}{\proof[\proofname]}{\endproof}

\numberwithin{equation}{subsection}

\begin{document}
\title
{Geometric polarized log Hodge structures with a base of log rank one}
%running title {Geometric polarized log Hodge structures}
\author{
Taro Fujisawa and Chikara Nakayama
}

\date{}
\maketitle

\newcommand\Cal{\mathcal}
\newcommand\define{\newcommand}

\define\bG{\bold G}
\define\bZ{\bold Z}
\define\bC{\bold C}
\define\bR{\bold R}
\define\bQ{\bold Q}
\define\bN{\bold N}
\define\bP{\bold P}
\define\cl{\mathrm{cl}}%
\define\fil{\mathrm{fil}}%
\define\fl{\mathrm{fl}}%
\define\gp{\mathrm{gp}}%
\define\fs{\mathrm{fs}}%
\define\an{\mathrm{an}}%
\define\et{\mathrm{et}}%
\define\mult{\mathrm{mult}}%
\define\sat{\mathrm{sat}}%
\define\Ker{\mathrm{Ker}\,}%
\define\Coker{\mathrm{Coker}\,}%
\define\Hom{\operatorname{\mathrm{Hom}}}%
\define\Aut{\operatorname{\mathrm{Aut}}}%
\define\Mor{\operatorname{\mathrm{Mor}}}%
\define\rank{\mathrm{rank}\,}%
\define\gr{\mathrm{gr}}%
\define{\cHom}{\operatorname{\mathcal{H}\mathit{om}}}
\define{\HOM}{\cHom}
\define{\cExt}{\operatorname{\mathcal{E}\mathit{xt}}}
\define{\cMor}{\operatorname{\mathcal{M}\mathit{or}}}
\define\cB{\Cal B}
\define\cO{\Cal O}
\renewcommand{\O}{\cO}
\define\cS{\Cal S}
\define\cM{\Cal M}
\define\cG{\Cal G}
\define\cH{\Cal H}
\define\cE{\Cal E}
\define\cF{\Cal F}

\newcommand{\ep}{\varepsilon}
\newcommand{\pe}{\frak p}
\newcommand{\Spec}{\operatorname{Spec}}
\newcommand{\val}{{\mathrm{val}}}
\newcommand{\bs}{\operatorname{\backslash}}
\newcommand{\Lie}{\operatorname{Lie}}

\renewcommand{\emph}{\it}

\newcommand{\ssm}{\smallsetminus}
\newcommand{\sig}{\sigma}
\newcommand{\Sig}{\Sigma}
\newcommand{\lan}{\langle}
\newcommand{\ran}{\rangle}
\newcommand{\fg}{\frak g}
\newcommand{\cD}{\check D}
\newcommand{\G}{\Gamma}
\define\Ext{\operatorname{\mathrm{Ext}}}%
\def\SL{\operatorname{SL}}
\def\Ad{\operatorname{Ad}}
\def\fsl{{\frak s\frak l}}
\renewcommand\Im{\mathrm{Im}\,}%
\define\End{\operatorname{\mathrm{End}}}%
\newcommand{\ket}{\mathrm{k\acute{e}t}}
\newcommand{\klog}{\mathrm{klog}}

%%%%%%%% definition by T. F.  %%%%%%%%%%%%%%

\define{\cmh}{cohomological mixed Hodge }
\define{\kos}{\operatorname{\mathrm{Kos}}}
\define{\bnnZ}{{\bold Z}_{\ge 0}}
\define{\id}{\operatorname{\mathrm{id}}}
\define{\e}{\mathbf e}
\define{\limind}{\operatorname*{\varinjlim}}
\newcommand{\RomII}{\uppercase\expandafter{\romannumeral 2}\ }
\newcommand{\RomIV}{\uppercase\expandafter{\romannumeral 4}\ }
\define{\cC}{\Cal C}
\define{\ula}{\underline{\lambda}}
\define{\dlog}{\operatorname{dlog}}
\define{\tr}{\operatorname{Tr}}

%%%%%%%%%%%%%%%%%%%%%%%%%%%%%%%%%%%%%%%%%%%%%%%%%%%

\begin{abstract}
  We prove that a projective vertical exact log smooth morphism of fs log analytic spaces with a base of log rank one yields polarized log Hodge structures in the canonical way.
\end{abstract}
\renewcommand{\thefootnote}{\fnsymbol{footnote}}
\footnote[0]{\hspace{-18pt}2010 {\it Mathematics Subject Classification}. Primary 32S35; Secondary 14D07, 32G20\\
{\it key words}. Hodge theory, log geometry, log Hodge structure}

{\bf Contents}

\medskip

\S\ref{s:plh}. Polarized log Hodge structures

\S\ref{s:mt}. Main theorems

\S\ref{s:rev}. Review of \cite{FN2}

\S\ref{s:sss}. Compatibility of polarizations

\S\ref{s:dege}. Degeneration of log Hodge and log de Rham spectral sequences

\S\ref{s:pf}.  Proof of Theorem \ref{t:main}

\S\ref{s:alb}.  Log Picard varieties and log Albanese varieties

\bigskip

\section*{Introduction}
  We prove that a projective vertical exact log smooth morphism $X \to S$ of fs log analytic spaces with a base of log rank one yields polarized log Hodge structures.

  In the previous paper \cite{FN2}, we treated the case where the base $S$ is the standard log point and $X$ is strict semistable over $S$. 
  See the introduction of \cite{FN2} for a history of this fundamental problem in log Hodge theory. 
  Notice here that the result of \cite{FN2} is the first affirmative one to this problem in the case where the base is not necessarily log smooth over $\bC$. 
  The present paper studies the case of more general base and more general family.  
  Even in the case where the base is the standard log point, the result in the present paper is far more general than that of \cite{FN2} in the sense that $X$ is any projective vertical log smooth family. 
  In particular, it is not necessarily saturated (in the classical term, it admits the multiplicities etc.), and we need to work with k\'et (= kummer log \'etale) formulation. 

  Roughly speaking, the proof is a reduction to \cite{FN2} by a variant of semistable reduction theorem. 
  We have to confirm many compatibilities and invariance of data in varying the base. 
  So, even the case where the base is the standard log point requires more general base cases. 

  The organization of the paper is as follows. 
  After a review of necessary definitions in Section \ref{s:plh}, we state the main results in Section \ref{s:mt}. 
  The most technical part of the paper is Section \ref{s:sss}, where we enhance 
the result in \cite{FN2} (Section \ref{s:rev} is a review of the necessary part of \cite{FN2}) 
by taking care of compatibility of polarizations. 
  Section \ref{s:dege} is a proof of the degeneration of log Hodge and log de Rham spectral sequences, which holds for general base. 
  Section \ref{s:pf} is a proof for that the family yields polarized log Hodge structures under the assumption that the log rank of the base is zero or one. 
  Section \ref{s:alb} is an application to log Picard varieties and log Albanese varieties.

\bigskip

\noindent {\sc Acknowledgments.} 
The authors thank to Luc Illusie for valuable advice.
The first author is partially supported by JSPS, Kakenhi (C) No.\ 16K05107.
The second author is partially supported by JSPS, Kakenhi (C) No.\ 16K05093.
The second author thanks J.\ C.\ for suggesting this work. 

\bigskip

\noindent {\sc Notation and Terminology.} 
  For the basic notions on log geometry in the analytic context, see \cite{KN} and \cite{KU}. 
  For an analytic space $S$, we denote by ${\cal O}_S$ the sheaf of 
holomorphic functions over $S$. 
  For an fs log analytic space $S$, we denote by $M_S$ the log structure on 
it. 
  For the associated morphism of ringed spaces $\tau\colon (S^{\log},\cO_S^{\log}) \to (S,\cO_S)$ 
(resp.\ the associated morphism of ringed spaces $\tau^{\ket}\colon (S^{\log},\cO_S^{\klog}) \to (S,\cO_S)$,
resp.\ the associated morphism of ringed topoi
$\varepsilon \colon 
(S_{\ket}, \cO_S^{\ket}) \to (S, \cO_S)$),
  see \cite{KN} Sections 1 and 3, and \cite{KU} 2.2.3--2.2.4 
(resp.\ \cite{IKN} Section 3,
resp.\ \cite{IKN} Section 2). 

  For a monoid or a sheaf of monoids $M$, we denote by $M^{\times}$ the subgroup of the invertible elements and 
by $\overline M$ the quotient $M/M^{\times}$.

  A morphism of fs log analytic spaces $f\colon X \to S$ is called {\it vertical} 
if the quotient monoid of $\overline M_{X,x}$ by the image of $\overline M_{S,f(x)}$ is a group for any $x \in X$.
  It is called {\it projective} 
if its underlying morphism of analytic spaces is so.

  An fs log analytic space is called an {\it fs log point} if its underlying analytic space is $\Spec \bC$. 
  It is called the {\it standard log point} if it is isomorphic to $(\Spec \bC, \bN \oplus \bC^{\times})$.
  Here the structural homomorphism $\bN \oplus \bC^{\times} \to \bC$ of the log structure sends 
the generator $1$ of $\bN$ to zero.

  For simplicity, we often denote a pullback of a sheaf on a space by the same symbol 
as that for the original sheaf. 

\section{Polarized log Hodge structure}\label{s:plh}
  We review the definition of polarized log Hodge structures in the k\'et sense 
(cf.\ \cite{KMN} 2.3).
  We first recall the nonk\'et version, which is also used in the following sections. 

\begin{para}
\label{triple}
  Let $S$ be an fs log analytic space. 
  Consider a triple $(L_{\bZ}, (F^p)_{p\in \bZ}, (\ ,\ ))$ 
consisting of a unipotent locally constant sheaf of free $\bZ$-modules $L_{\bZ}$ on $S^{\log}$, of a decreasing filtration 
$(F^p)_{p\in \bZ}$ of an $\cO_S$-module $L_{\cal O}$ endowed with an isomorphism from it to $\tau_*(L_{\bZ} \otimes \cO_S^{\log})$ by sub $\cO_S$-modules which are locally direct summands, and of a nondegenerate $\bQ$-bilinear form $(\ ,\ )$ on $L_{\bQ}=\bQ \otimes L_{\bZ}$.
  Below we identify $(\tau^*L_{\cal O}) \otimes \cO_S^{\log}$ and $L_{\bZ} \otimes \cO_S^{\log}$ via the given isomorphism (cf.\ \cite{KU} {\sc Proposition} 2.3.3 (i)).
\end{para}

\begin{para}
\label{plh}
  Let $w$ be an integer. 
  If a triple as in \ref{triple} satisfies the following condition (1), we call it a {\it nonk\'et prePLH (= pre-polarized log Hodge structure) on $S$ of weight $w$}. 

(1) (Orthogonality.) $(\tau^*F^p,\tau^*F^q)=0$ if $p+q>w$. 
  Here $(\ ,\ )$ is regarded as the natural extension to $\cO_S^{\log}$-bilinear form. 

  If a nonk\'et prePLH on $S$ of weight $w$ satisfies the following conditions (2) and (3) after pulling back to each point $s \in S$, we call it a {\it nonk\'et PLH (= polarized Hodge structure) on $S$ of weight $w$}. 
(It is simply called PLH on $S$ of weight $w$ in \cite{KU} 2.4.) 

(2) (Positivity.) Assuming that $S$ is an fs log point with the unique point $s$, let $t \in S^{\log}$ and $h \in \Hom_{\bC\text{-}{\mathrm{alg}}}(\cO_{S,t}^{\log}, \bC)$. 
  If the induced homomorphism $$M_{S,s} \to \bC; a \mapsto \exp(h(\log(a)))$$ 
is sufficiently near to the structural homomorphism of the log structure with respect to the topology of simple convergence of $\bC$-valued functions, then the specialization 
$$(L_{\bZ, t}, (F^p(h))_{p \in \bZ}, (\ ,\ )_t)$$ 
of the triple by $h$ is a polarized Hodge structure of weight $w$ in the usual sense. 

(3) (Griffiths transversality.) 
  Assuming that $S$ is an fs log point, the image of 
$F^p$ by $d \otimes 1 \colon \cO_S^{\log} \otimes L_{\bZ} \to 
\omega_S^{1,\log}\otimes L_{\bZ}$ is contained in $\omega_S^{1,\log}\otimes F^{p-1}$ for any $p$.
\end{para}

\begin{para}
\label{ket_triple}
  Next we define the k\'et version. 
  Let $S$ be an fs log analytic space. 
  Consider a triple $(L_{\bZ}, (F^p)_{p\in \bZ}, (\ ,\ ))$ 
consisting of a quasi-unipotent locally constant sheaf of free $\bZ$-modules $L_{\bZ}$ on $S^{\log}$, of a decreasing filtration 
$(F^p)_{p\in \bZ}$ of an $\cO_S^{\ket}$-module $L_{{\cal O}^{\ket}}$ endowed with an isomorphism from it to 
$\tau^{\ket}_*(L_{\bZ} \otimes \cO_S^{\klog})$ 
by sub $\cO_S^{\ket}$-modules which are locally direct summands, and of a nondegenerate $\bQ$-bilinear form $(\ ,\ )$ on $L_{\bQ}=\bQ \otimes L_{\bZ}$.
  Below we identify $(\tau^*L_{{\cal O}^{\ket}}) \otimes \cO_S^{\klog}$ and $L_{\bZ} \otimes \cO_S^{\klog}$ via the given isomorphism (cf.\ the beginning of the proof of \cite{IKN} {\sc Theorem} (4.4)).
\end{para}

\begin{para}
  Let $w$ be an integer. 
  If a triple as in \ref{ket_triple} satisfies the following condition (1), we call it a {\it prePLH on $S$ of weight $w$}. 

(1) (Orthogonality.) $(\tau^{\ket*}F^p,\tau^{\ket*}F^q)=0$ if $p+q>w$. 
  Here $(\ ,\ )$ is regarded as the natural extension to $\cO_S^{\klog}$-bilinear form. 

  If a prePLH on $S$ of weight $w$ satisfies the following conditions (2) and (3) after pulling back to each point $s \in S$, we call it a {\it PLH on $S$ of weight $w$}. 

(2) (Positivity.) Assuming that $S$ is an fs log point with the unique point $s$, let $t \in S^{\log}$ and $h \in 
\Hom_{\bC\text{-}{\mathrm{alg}}}(\cO_{S,t}^{\klog}, \bC)=\Hom_{\bC\text{-}{\mathrm{alg}}}(\cO_{S,t}^{\log}, \bC)$.
  If the induced homomorphism $$M_{S,s} \to \bC; a \mapsto \exp(h(\log(a)))$$ 
is sufficiently near to the structural homomorphism of the log structure with respect to the topology of simple convergence of $\bC$-valued functions, then the specialization 
$$(L_{\bZ, t}, (F^p(h))_{p \in \bZ}, (\ ,\ )_t)$$ 
of the triple by $h$ is a polarized Hodge structure of weight $w$ in the usual sense. 

(3) (Griffiths transversality.) 
  Assuming that $S$ is an fs log point, the image of 
$F^p$ by $d \otimes 1 \colon \cO_S^{\klog} \otimes L_{\bZ} \to 
\omega_S^{1,\klog}\otimes L_{\bZ}$ is contained in $\omega_S^{1,\klog}\otimes F^{p-1}$ for any $p$.
\end{para}

\begin{para}
\label{relation}
  The relation between the nonk\'et versions and k\'et versions is as follows. 

  Let $S$ and $w$ be as in the above. 
  The pullback by the forgetting log morphism 
$$\varepsilon \colon (S_{\ket}, \cO_S^{\ket}) \to (S, \cO_S)$$
of a nonk\'et prePLH (resp.\ nonk\'et PLH) on $S$ of weight $w$ is a 
prePLH (resp.\ PLH) on $S$ of weight $w$. 
  This gives a category equivalence between the category of the nonk\'et ones to the full subcategory of the category 
of the k\'et ones whose lattice $L_{\bold Z}$ has unipotent local monodromies (cf.\ \cite{KMN} Remark 2.4). 
\end{para}

\section{Main theorems}\label{s:mt}
  In this section, we describe the main results of this paper. 

\begin{para}
\label{pairing}
  Let $f\colon X \to S$ be a projective, vertical, exact and log smooth morphism of fs log analytic spaces. 

  Fix an invertible $\cO_X$-module $L$ which is relatively ample with respect to $S$. 
  Let $c(L)$ be the image of the class of $L$ by 
$$R^1f_*\cO_X^* \to R^2f_*\bZ(1) \to R^2f^{\log}_*\bZ(1).$$

  Let $q$ be an integer, and let 
$$H^q_{\bZ}=R^qf^{\log}_*\bZ, \quad 
  H^q_{\bQ}=R^qf^{\log}_*\bQ.$$
  We define a $\bQ$-bilinear pairing 
$(\ ,\ )\colon H^q_{\bQ} \times H^q_{\bQ} \to \bQ$. 

  By taking cup product with $(2\pi \sqrt{-1})^{-1}c(L)$,
a morphism
$$
l=(2\pi \sqrt{-1})^{-1}c(L) \cup :
H^q_{\bQ} \to H^{q+2}_{\bQ}
$$
is defined for all $q$.

  Let $n$ be the relative dimension of $X$ over $S$, which is a locally constant function (see Proposition \ref{p:reldim} below). 
  Assume $q \le n$, where $q$ is regarded as a constant function on $S$. 
  Let $H^q_{\bQ,\mathrm{prim}}$ be the kernel of 
$$l^{n-q+1} \colon H^q_{\bQ} \to H^{2n-q+2}_{\bQ}.$$ 
  Here, $H^{2n-q+2}_{\bQ}$ means the sheaf of $\bQ$-vector spaces which is isomorphic to $H^{2n(i)-q+1}_{\bQ}$ on each connected component $S_i$ of $S$, where $n(i)$ is the value of $n$ on $S_i$. 
  The meaning of $l^{n-q+1}$ is similar. 
\end{para}

  The next theorem will be proved in Section \ref{s:dege}. 

\begin{thm}[Lefschetz decomposition]
\label{t:Lef}
  The homomorphism 
$$
\bigoplus_j H^j_{\bQ,\mathrm{prim}} \to H^q_{\bQ};
(a_j)_j \mapsto \sum_j l^{\frac{q-j}2}(a_j),
$$
where $j$ ranges over all integers such that $j \le n, j \le q$, and $j \equiv q\bmod 2$, is an isomorphism. 
\end{thm}

By this theorem, we can define $(\ ,\ )$ in the usual way (cf.\ \cite{KMN} 8.2), 
that is, we define it as the unique $\bQ$-bilinear form such that the subspaces 
$l^{\frac{q-j}2}H^j_{\bQ,\mathrm{prim}}$
and $l^{\frac{q-k}2}H^k_{\bQ,\mathrm{prim}}$
of $H^q_{\bQ}$ for $j, k$ as above are orthogonal to each other under $(\ , \ )$
if $j \not= k$,
and
$$
(l^{\frac{q-j}2}(a), l^{\frac{q-j}2}(b))
$$
for $j$ as above and for $a, b \in H^j_{\bQ,\mathrm{prim}}$ is the image of $(-1)^{\frac{j(j-1)}2} a \otimes b$ under 
$$H^j_{\bQ} \otimes H^j_{\bQ} \to H^{2j}_{\bQ} \to H^{2n}_{\bQ} \to \bQ,$$
where the first map is the cup product, the second is induced by
$l^{n-j}$, 
and the last one is the trace map 
(cf.\ \cite{NO} Section 5).

  We also have the following theorem.
  This is an analytic version of \cite{IKN} {\sc Theorem} $7.1$. 

\begin{thm}
\label{t:HdRss}
  Let $f\colon X \to S$ be a projective, vertical, exact and log smooth morphism of fs log analytic spaces. 
  Then the following hold. 

$(1)$ The log Hodge to log de Rham spectral sequence 
$$E_1^{pq} = R^qf_{\ket*}\omega^{p}_{X/S,\ket} \Rightarrow 
R^{p+q}f_{\ket*}\omega^{\cdot}_{X/S,\ket}$$
degenerates from $E_1$. 

$(2)$ The sheaves $R^qf_{\ket*}\omega^{p}_{X/S,\ket}$ for all $p, q$ are locally free of finite type over $S_{\ket}$ and commute with base change by any morphism of fs log analytic spaces $S' \to S$. 
\end{thm}

\begin{sbrem} 
(1)  Without exactness, the conclusion of Theorem \ref{t:HdRss} (2) fails. 
  A counter example is as follows. 
  Let $S=\Spec \bC[\bN^2]/(x^2, y^2)$, where $x,y$ is the standard basis of $\bN^2$. 
  Let $X$ be the log blow-up (\cite{IKN} {\sc Definition} (6.1.1))
of $S$ with respect to the ideal generated by $x$ and $y$. 
  Then the natural homomorphism $\Gamma(S,\O_S) \to \Gamma(X,\O_X)$ is not injective so that $f_{\ket*}{\cal O}_S^{\ket}$ is not locally free. 

(2)  As is well-known, we cannot replace the projectivity in this theorem by the properness even in the classical case 
(an Iwasawa manifold over a point is a counter example (\cite{Nakamura}, \cite{Sakane})).
  On the other hand, it is plausible that we can weaken projectivity by log K\"ahlerness (cf.\ \cite{KKN} Section 9). 
\end{sbrem}

\begin{cor}
  Under the assumption in Theorem $\ref{t:HdRss}$, the following hold. 

$(1)$ The nonk\'et log Hodge to log de Rham spectral sequence 
$$E_1^{pq} = R^qf_{*}\omega^{p}_{X/S} \Rightarrow 
R^{p+q}f_{*}\omega^{\cdot}_{X/S}$$
degenerates from $E_1$. 

$(2)$ Assume that the stalk of $(M_S/\cO^{\times}_S)_s$ for any $s\in S$ is a free monoid.
  Then the sheaves $R^qf_{*}\omega^{p}_{X/S}$ for all $p, q$ are locally free of finite type over $S$. 
\end{cor}

  This is deduced from Theorem \ref{t:HdRss} as in the same way as \cite{IKN} {\sc Corollary} (7.2) is deduced from \cite{IKN} {\sc Theorem} (7.1).

  The next is a variant of this corollary. 

\begin{thm}
\label{t:variant}
  Let the notation and the assumption be as in Theorem $\ref{t:HdRss}$. 
  Assume further that $f$ is saturated. 
  Then the following hold. 

$(1)$ The nonk\'et log Hodge to log de Rham spectral sequence 
$$E_1^{pq} = R^qf_{*}\omega^{p}_{X/S} \Rightarrow 
R^{p+q}f_{*}\omega^{\cdot}_{X/S}$$
degenerates from $E_1$. 

$(2)$ The sheaves $R^qf_{*}\omega^{p}_{X/S}$ for all $p, q$ are locally free of finite type over $S$ and commute with base change by any morphism of fs log analytic spaces $S' \to S$. 

$(3)$ The (k\'et) log Hodge to log de Rham spectral sequence in Theorem $\ref{t:HdRss}$ $(1)$ is the pullback of the nonk\'et one in the above $(1)$. 
\end{thm}

  The following is a goal of this paper. 

\begin{thm}
\label{t:main}
  Let $f\colon X \to S$ be a projective, vertical, exact and log smooth morphism of fs log analytic spaces. 
  Assume that the log rank of $S$ is $\leq 1$, that is, $(M_S/\cO_S^{\times})_s$ is either trivial or isomorphic to $\bN$ for each $s \in S$. 
  Then, for each $q \in \bZ$,  
$$\bigl(H^q_{\bZ}, (R^qf_{\ket*}\omega^{\cdot\ge p}_{X/S,\ket})_{p \in \bZ}, (\ , \ )\bigr)$$
is a PLH on $S_{\ket}$.
\end{thm}

  Here, $(R^qf_{\ket*}\omega^{\cdot\ge p}_{X/S,\ket})_{p \in \bZ}$ is regarded as a filtration on $H^q_{{\cal O}^{\ket}}:= 
R^qf_{\ket*}\omega^{\cdot}_{X/S,\ket}$ by Theorem \ref{t:HdRss} (1) and this $H^q_{{\cal O}^{\ket}}$ is endowed with a natural isomorphism from it to $\tau^{\ket}_*(H^q_{\bZ} \otimes \cO_S^{\klog})$ by \cite{IKN} {\sc Theorem} (6.2) (3). 

\begin{sbrem}
  It has been conjectured that, in the theorem, the assumption that the log rank of $S$ is $\leq 1$ is not necessary. 
  The proof below shows that this conjecture is valid in the case where the base is an fs log point.
  The latter case with the multisemistable reduction probably can be reduced to \cite{Fmultiss}. 
  But it seems that the general case requires a new idea. 
  See \ref{difficulty} for more discussions.
  Further, it is again plausible that we can weaken  projectivity by log K\"ahlerness.
\end{sbrem}

\begin{cor}
\label{c:main}
  Let the notation and the assumption be as in Theorem $\ref{t:main}$. 
  Assume further that $f$ is saturated. 
  Then, for each $q \in \bZ$,  
$$\bigl(H^q_{\bZ}, (R^qf_{*}\omega^{\cdot\ge p}_{X/S})_{p \in \bZ}, (\ , \ )\bigr)$$
is a nonk\'et PLH on $S$. 
\end{cor}

  Here, $(R^qf_{*}\omega^{\cdot\ge p}_{X/S})_{p \in \bZ}$ is regarded as a filtration on $H^q_{{\cal O}}:= 
R^qf_{*}\omega^{\cdot}_{X/S}$ by Theorem \ref{t:variant} (1) and this $H^q_{{\cal O}}$ is endowed with a natural isomorphism from it to $\tau_*(H^q_{\bZ} \otimes \cO_S^{\log})$ by \cite{IKN} {\sc Theorem} (6.3) (3) together with the fact that 
the assumption that $f$ is saturated implies that 
the cokernel of $(M^{\gp}_S/\cO^{\times}_S)_{f(x)} \to (M^{\gp}_X/\cO^{\times}_X)_{x}$ is torsion-free for any $x \in X$ (cf.\ \cite{IKN} {\sc Lemma} (A.4.1)). 
 (In fact, the converse of the last fact holds. 
  See Proposition \ref{sp:sat} below.)

  Taking the above fact into account, Corollary \ref{c:main} is deduced from Theorem \ref{t:main} as in the same way as \cite{KMN} Theorem 8.11 (3) is deduced from \cite{KMN} Theorem 8.1 (2).  
  We use \ref{relation} and the fact that $H^q_{\bZ}$ has a unipotent local monodromy in each point, which is by \cite{IKN} {\sc Theorem} (6.3) (1). 

\begin{sbprop}
\label{sp:sat}
  Let $f\colon X \to S$ be a log smooth and exact morphism of fs log analytic spaces.  
  Assume that the cokernel of $(M^{\gp}_S/\cO^{\times}_S)_{f(x)} \to (M^{\gp}_X/\cO^{\times}_X)_{x}$ is torsion-free for any $x \in X$ . 
  Then $f$ is saturated.
\end{sbprop}

\begin{pf}
  Let $s \in S$. 
  Take an fs chart $h\colon P \to Q$ for $f$ around $s$ such that $P \overset \cong \to (M_S/\cO^{\times}_S)_s$ and that $Q$ is torsion-free (cf.\ the proposition in \cite{N:near} (A.2)). 
  We may assume that $h$ is $\bQ$-integral by the proof of \cite{IKN} {\sc Proposition} (A.3.3). 
  Using the assumption with the points lying over $s$, we see that, for any face $F$ of $Q$ such that $h^{-1}(F)=P^{\times}$, the cokernel of $P^{\gp} \to (Q/F)^{\gp}$ is torsion-free. 
  Then by Ogus' characterization \cite{O} Chapter I, Theorem 4.8.14 of saturated homomorphisms, we conclude that $h$ is saturated. 
  Hence, $f$ is also saturated. 
\end{pf}

\smallskip

  In the rest of this section, we prove the easy fact on the relative dimension, which is used in \ref{pairing}. 

\begin{prop}
\label{p:reldim}
  Let $f\colon X \to S$ be a proper log smooth and exact morphism of fs log analytic spaces. 
  Then the function $S \ni s \mapsto \dim(f^{-1}(s))$ is a locally constant function. 
\end{prop}

\begin{pf}
  We may assume that both $X$ and $S$ are connected. 
  Then $f$ is surjective because it is proper. 
  Further, since each fiber is compact, taking a local, $\bQ$-integral and injective chart (cf.\ \cite{IKN} {\sc Proposition (A.3.3)}), we reduce to the following statement. 
\end{pf}

\begin{prop}
\label{p:equidim}
  Let $h \colon P \to Q$ be a local, $\bQ$-integral and injective homomorphism of fs monoids. 
  Let $k$ be a field. 
  Let $d:= \rank_{\bZ}(Q^{\gp}) - \rank_{\bZ}(P^{\gp})$. 
  Then any fiber of $\Spec k[h]$ is connected and equi-$d$-dimensional. 
\end{prop}

\begin{pf}
  It is a variant of \cite{N:le2} Proposition 7.13 and proved similarly as follows. 
  By localizing $P$ and replacing $P$ by $\overline P$ and $Q$ by $Q/h(P)$, we see that it is enough to show that, for a sharp $P$, the fiber of the vertex of $\Spec k[P]$ is connected and equi-$d$-dimensional. 
  By the proof of \cite{N:le2} Proposition 7.13, this fiber is homeomorphic to the subspace 
$$\bigcup_F \Spec k[F]$$
of $\Spec k[Q]$. 
  Here $F$ runs over the set of all faces of $Q$ whose intersection of the image of $P$ is trivial. 
  This subspace is connected and 
equi-$d$-dimensional by \cite{N:le2} Proposition 7.15. 
  This completes the proof of Propositions \ref{p:reldim} and \ref{p:equidim}.
\end{pf}

\section{Review of \cite{FN2}}
\label{s:rev}
  In this section, we review the main result of \cite{FN2}, which is a variant of a special case of Corollary \ref{c:main} and to which we reduce our main results. 

\begin{defn}
  Let $S$ be the standard log point. 
  An fs log analytic space $X$ over $S$ is called a {\it log deformation} if it is locally isomorphic over $S$ to an open sub log analytic space of the log special fiber of the standard semistable family $\Spec \bC[\bN \to \bN^r; 1 \mapsto (1,1,\ldots,1)]$ for some $r$.  
  See \cite{S95} Definition (3.8). 
  A {\it strict log deformation} is a log deformation whose irreducible components of the underlying analytic space are all smooth. 
\end{defn}

\begin{thm}
\label{t:FN2}
  Let $f\colon X \to S$ be a projective strict log deformation over the standard log point. 
  Then, for each $q \in \bZ$, 
$$\bigl(H^q_{\bZ}, (R^qf_{*}\omega^{\cdot\ge p}_{X/S})_{p \in \bZ}\bigr)$$
is a log Hodge structure of weight $q$ on $S$.
\end{thm} 

  This is {\sc Theorem} 2.7 in \cite{FN2}.
  See the next remark. 

\begin{rem}
  There is a trivial mistake in \cite{FN2}.  
  In the statement of the main theorem {\sc Theorem} 2.7 (also that of its vague form {\sc Corollary 1.2}) in \cite{FN2}, \lq\lq log deformation'' should be replaced by \lq\lq strict log deformation.''
  This strictness is essential in the calculation in Section 3 of \cite{FN2} and should be imposed throughout the paper \cite{FN2}. 
\end{rem}

  The log Hodge structure given by Theorem \ref{t:FN2} is polarized by \cite{F} via the cohomologies of the Koszul complex. 
  More precisely, we have proved Theorem \ref{t:FN2} in \cite{FN2} by constructing the comparison isomorphism of $\bQ$-structures 
$$H^q_{\bQ,z} \cong H^q(X,K_{\bQ}),$$
and by reducing to the main result in \cite{F} that the right-hand-side carries a polarized log Hodge structure. 
  Here the notation is as in \cite{FN2}. 
  In particular, $z$ is the point of $0^{\log}$ corresponding to $1$, 
where we identify the set $0^{\log}$ with the set 
$\{\alpha \in \bC\,\vert\, \alpha \overline \alpha = 1\}$.

\section{Compatibility of polarizations}\label{s:sss}

In this section, we prove Theorem \ref{t:Lef}, the nonk\'et version of 
Theorem \ref{t:HdRss} (i.e., Theorem \ref{t:variant} (1), (2)),  and the nonk\'et version of Theorem \ref{t:main} (i.e., Corollary \ref{c:main}) for a strict log deformation over the standard log point. 
  Throughout this section, $S$ is the standard log point and $f\colon X \to S$ is a projective strict log deformation. 

\begin{para}
\label{para: three compatibilities}
As explained in Section \ref{s:rev}, for such an $f$, 
$$\bigl(H^q_{\bZ}, (R^qf_{*}\omega^{\cdot\ge p}_{X/S})_{p \in \bZ}\bigr)$$
is a log Hodge structure on $S$, and this log Hodge structure is polarized by \cite{F} via the cohomologies of the Koszul complex. 
  Hence, the compatibility of polarizations reduces 
Theorem \ref{t:Lef}, Theorem \ref{t:variant} (1), (2),  and Corollary \ref{c:main} for $f$ to the corresponding results on the cohomologies of the Koszul complex. 
  Furthermore, this is reduced to the compatibility of Lefshetz operator, cup product, and the trace map. 
  More precisely, what we should prove is the following.
  As is explained in Section \ref{s:rev}, we have the comparison isomorphism of $\bQ$-structures 
$$H^q_{\bQ,z} \cong H^q(X_{\infty}^{\log}, \bQ) \cong H^q(X,K_{\bQ}).$$

Then we should have
\begin{itemize}
\item[(1)]
The germ of $c(L)$ is sent to
$-(2\pi\sqrt{-1})c_{K,\bQ}(L) \in H^2(X,K_{\bQ})$.
\item[(2)]
For $p,q \in \bZ$, the stalk of the cup product
$H^p_{\bQ} \otimes H^q_{\bQ} \to H^{p+q}_{\bQ}$ is compatible with 
the cup product $H^p(X,K_{\bQ}) \otimes H^q(X,K_{\bQ}) \to H^{p+q}(X,K_{\bQ})$.
\item[(3)]
The stalk of the trace map $H^{2n}_{\bQ} \to \bQ$ is compatible with 
$\tr:H^{2n}(X,K_{\bQ}) \to \bQ$.
\end{itemize}
These are reduced to the $\bC$-versions,
which is to be proved in the remainder of this section.
\end{para}

\begin{rem}
Notice that the minus sign in front of
$(2\pi\sqrt{-1})c_{K,\bQ}(L)$ in (1) above
comes from Definition 8.5 of \cite{F}.
\end{rem}

\begin{para}
\label{para: identifications}
In order to prove the compatibility above,
we briefly recall the identifications
\begin{equation}
\label{eq:identification1}
\bC \otimes H^q_{\bQ,z}
\cong
H^q(X_{\infty}^{\log}, \bC)
\cong
H^q(X, \omega^{\cdot}_{X/0})
\cong
H^q(X,K_{\bC})
\end{equation}
in \cite{FN2}.
We have the commutative diagram of ringed spaces: 
$$
\begin{CD}
X^{\log}_z @>i>> X^{\log}_{\infty} @>\pi>> X^{\log} @>\tau>> X \\
@VV f^{\log} V @VVf_{\infty}V @VV f^{\log} V @VV f V  \\
\{0\} @>i>> \bR @>\pi>> 0^{\log} @>\tau>> 0, 
\end{CD}$$
where the bottom $\pi$ is the map
$\bR \ni r \mapsto e^{2 \pi i r} \in 0^{\log}$, and 
the left two squares are cartesian.
Note that $X^{\log}_z$ is nothing but
the fiber of $f^{\log}: X^{\log} \to 0^{\log}$
over the point $z \in 0^{\log}$
and that the morphism
$\pi i: X^{\log}_z \to X^{\log}$ is the canonical closed immersion.
As proved in 7.8 and 7.10 of \cite{FN2},
the identifications
\begin{equation}
\label{eq:identification2}
\bC \otimes H^q_{\bQ,z}
\cong
H^q(X_{\infty}^{\log}, \bC)
\cong
H^q(X, \omega^{\cdot}_{X/0})
\end{equation}
coincide with the isomorphism
obtained from the sequence of morphisms
$$
R(\tau \pi i)_*\bC
\gets
R\tau_*\bC
\to
R(\tau \pi)_*\pi^{-1}\omega^{\cdot, \log}_X
\gets
\omega^{\cdot}_X[u]
\to
\omega^{\cdot}_{X/0},
$$
where the morphisms are induced from
the canonical morphisms
$\bC \to i_*\bC$,
$\pi^{-1}\bC=\bC \to \pi^{-1}\omega^{\cdot, \log}_X$,
the morphism $\tau^{-1}\omega^{\cdot}_X[u] \to \omega^{\cdot,\log}_X$
sending $u$ to $(2\pi\sqrt{-1})^{-1}\log t$
and the morphism $\omega^{\cdot}_X[u] \to \omega^{\cdot}_{X/0}$
sending $u$ to $0$.
Since the commutative diagram
$$
\begin{CD}
\bC @>>> \pi^{-1}\omega^{\cdot,\log}_X
    @<<< (\tau\pi)^{-1}\omega^{\cdot}_X[u]
    @>>> (\tau\pi)^{-1}\omega^{\cdot}_{X/0} \\
@VVV @VVV @VVV @VVV \\
i_*\bC @>>> i_*(\pi i)^{-1}\omega^{\cdot,\log}_X
       @<<< i_*(\tau\pi i)^{-1}\omega^{\cdot}_X[u]
       @>>> i_*(\tau\pi i)^{-1}\omega^{\cdot}_{X/0}
\end{CD}
$$
yields the commutative diagram
$$
\begin{CD}
R\tau_*\bC @>>> R(\tau\pi)_*\pi^{-1}\omega^{\cdot,\log}_X
    @<<< \omega^{\cdot}_X[u]
    @>>> \omega^{\cdot}_{X/0} \\
@VVV @VVV @| @| \\
R(\tau\pi i)_*\bC @>>> R(\tau\pi i)_*(\pi i)^{-1}\omega^{\cdot,\log}_X
       @<<< \omega^{\cdot}_X[u]
       @>>> \omega^{\cdot}_{X/0},
\end{CD}
$$
the identification \eqref{eq:identification2}
coincides with the isomorphism
induced from the sequence of the morphisms
$$
R\Gamma(X^{\log}_z, \bC)
\to
R\Gamma(X^{\log}_z, (\pi i)^{-1}\omega^{\cdot,\log}_X)
\gets
R\Gamma(X, \omega^{\cdot}_X[u])
\to
R\Gamma(X, \omega^{\cdot}_{X/0}).
$$
On the other hand,
the identification
$$
H^q(X, \omega^{\cdot}_{X/0}) \cong H^q(X, K_{\bC})
$$
is obtained from the commutative diagram
$$
\begin{CD}
\omega^{\cdot}_X[u] @>>> K_{\bC} \\
@VVV @VVV \\
\omega^{\cdot}_{X/0} @>>> \cC(\omega^{\cdot}_{X_{\bullet}/0})
\end{CD}
$$
as in the proof of Theorem 5.29 of \cite{F}.
Combining these together,
the commutative diagram
\begin{equation}
\label{eq:identification3}
\begin{CD}
R\Gamma(X^{\log}_z, \bC) \\
@VVV \\
R\Gamma(X^{\log}_z, (\pi i)^{-1}\omega^{\cdot,\log}_X) \\
@AAA \\
R\Gamma(X, \omega^{\cdot}_X[u])
@>>>
R\Gamma(X, K_{\bC}) \\
@VVV @VVV \\
R\Gamma(X, \omega^{\cdot}_{X/0})
@>>>
R\Gamma(X, \cC(\omega^{\cdot}_{X_{\bullet}/0}))
\end{CD}
\end{equation}
induces the identifications \eqref{eq:identification1}.
\end{para}

\begin{prop}
The image of the germ of the log Chern class $c(L)$ in $\bC \otimes H^2_{\bQ,z}$
coincides with $-c_{K,\bC}(L) \in H^2(X, K_{\bC})$
under the identification \eqref{eq:identification1}.
\end{prop}
\begin{proof}
The canonical inclusion $\bZ(1) \to \bC$
fits in the commutative diagram
$$
\begin{CD}
\bZ(1) @>>> R\tau_*\bZ(1) \\
@VVV @VVV \\
\bC @>>> R(\tau\pi i)_\ast\bC
\end{CD}
$$
on $X$,
by which it turns out that
the image of $c(L)$ in $\bC \otimes H^2_{\bQ,z} \cong H^2(X^{\log}_z, \bC)$
coincides with the image of the usual Chern class,
denoted by $c(L)$ again,
by the canonical morphism $H^2(X, \bC) \to H^2(X^{\log}_z, \bC)$.
On the other hand,
we have the commutative diagram
$$
\xymatrix{
& R\Gamma(X, \bC) \ar[ld] \ar[d] \ar[rd] \ar[rrd] & & & \\
R\Gamma(X^{\log}_z,\bC) \ar[r]
& R\Gamma(X^{\log}_z, (\pi i)^{-1}\omega^{\cdot,\log}_X)
& R\Gamma(X, \omega^{\cdot}_X[u]) \ar[l] \ar[r]
& R\Gamma(X, \omega^{\cdot}_{X/0})
}
$$
where the bottom row is nothing but the left vertical arrows
in \eqref{eq:identification3}.
Thus we have to compare the image of $c(L) \in H^2(X,\bZ(1))$
in $H^2(X, \omega^{\cdot}_{X/0})$
and $c_{K,\bC}(L) \in H^2(X, K_{\bC})$.
The restriction of the morphism
$\dlog: M^{\gp}_X \to \omega^1_X$ to $\cO^{\times}_X \subset M^{\gp}_X$
induces a morphism of complexes
$\cO^{\times}_X \to \omega^{\cdot}_X[1]$,
denoted by the same letter $\dlog$ again.
Then $c_{K,\bC}(L)$ is defined via the composite of the morphisms
$$
\cO^{\times}_X
\overset{\dlog}{\longrightarrow}
W_0\omega^{\cdot}_X[1]
\to \omega^{\cdot}_X[1]
\to K_{\bC}[1],
$$
in Definition 8.3 of \cite{F}.
Now we consider the diagram
$$
\begin{CD}
\cO^{\times}_X @>{-\dlog}>> W_0\omega^{\cdot}_X[1]
        @>>> \omega^{\cdot}_X[1]
        @>>> K_{\bC}[1] \\
@VVV @AAA @VVV @VVV \\
\bZ(1)[1]
        @>>> \bC[1]
        @>>> \omega^{\cdot}_{X/0}[1]
        @>>> \cC(\omega^{\cdot}_{X_{\bullet}/0})[1]
\end{CD}
$$
in the derived category,
where the first vertical arrow
is the morphism in the derived category
defined by the exponential exact sequence
$$
0 \to \bZ(1) \to \cO_X \to \cO^{\times}_X \to 0.
$$
It is easy to check that
the middle and right squares are commutative as the diagram of complexes.
The left square commutes in the derived category
by the same argument as in (2.2.5) of \cite{D3}.
From the bottom square of the identification \eqref{eq:identification3},
we have the desired coincidence.
\end{proof}

\begin{rem}
By Lemma 8.4 of \cite{F},
$c_{K,\bC}(L)$ coincides with $(2\pi\sqrt{-1})c_{K,\bQ}(L)$
under the identification
$\bC \otimes H^2(X, K_{\bQ}) \cong H^2(X, K_{\bC})$.
Thus the lemma above implies the compatibility (1)
in \ref{para: three compatibilities}.
\end{rem}

  The next shows the compatibility (2) in \ref{para: three compatibilities}.

\begin{prop}
Under the identification
\eqref{eq:identification1},
the cup product on
$\bC \otimes H^q_{\bQ,z}$
is compatible with the cup product on $H^q(X, K_{\bC})$.
\end{prop}
\begin{proof}
It is sufficient to prove that
the cup product on
$H^q(X^{\log}_z,\bC)$
is compatible with the cup product of $H^q(X,K_{\bC})$
under the isomorphisms in \eqref{eq:identification3}.
The cup product on $H^q(X^{\log}_z,\bC)$
is induced from the morphism
$\bC \otimes \bC \to \bC$
defined by $a \otimes b \mapsto ab$.
On the other hand,
the morphism of complexes
$\omega^{\cdot}_{X/0} \otimes \omega^{\cdot}_{X/0} \to \omega^{\cdot}_{X/0}$
is given by $\omega \otimes \eta \mapsto \omega \wedge \eta$.
For the complex
$\omega^{\cdot,\log}_X=\cO^{\log}_X \otimes \tau^{-1}\omega^{\cdot}_X$,
the morphism
$\omega^{\cdot,\log}_X \otimes \omega^{\cdot,\log}_X
\to \omega^{\cdot,\log}_X$
is defined by
\begin{equation}
\label{eq:cupproduct}
\omega^{p,\log}_X \otimes \omega^{q,\log}_X
\ni
(a \otimes \omega) \otimes (b \otimes \eta)
\mapsto
ab \otimes \omega \wedge \eta \in \omega^{p+q,\log}_X
\end{equation}
where $a,b \in \cO^{\log}_X$
and $\omega \in \tau^{-1}\omega^p_X, \eta \in \tau^{-1}\omega^q_X$.
It is trivial that this defines a morphism of complexes.
For the complex
$\omega^{\cdot}_X[u]=\bC[u] \otimes \omega^{\cdot}_X$,
the morphism
$\omega^p_X[u] \otimes \omega^q_X[u] \to \omega^{p+q}_X[u]$
is defined by
the same formula as in \eqref{eq:cupproduct}.
Then it is easy to check the commutativity of the diagrams
$$
\begin{CD}
\bC \otimes \bC @>>> \bC \\
@VVV @VVV \\
(\pi i)^{-1}\omega^{\cdot,\log}_X \otimes (\pi i)^{-1}\omega^{\cdot,\log}_X
@>>> (\pi i)^{-1}\omega^{\cdot,\log}_X
\end{CD}
$$
on $X^{\log}_z$,
$$
\begin{CD}
\omega^{\cdot,\log}_X \otimes \omega^{\cdot,\log}_X
@>>> \omega^{\cdot,\log}_X \\
@AAA @AAA \\
\tau^{-1}\omega^{\cdot}_X[u] \otimes \tau^{-1}\omega^{\cdot}_X[u]
@>>> \tau^{-1}\omega^{\cdot}_X[u]
\end{CD}
$$
on $X^{\log}$
and
$$
\begin{CD}
\omega^{\cdot}_X[u] \otimes \omega^{\cdot}_X[u]
@>>> \omega^{\cdot}_X[u] \\
@VVV @VVV \\
\omega^{\cdot}_{X/0} \otimes \omega^{\cdot}_{X/0}
@>>> \omega^{\cdot}_{X/0}
\end{CD}
$$
on $X$ for these morphisms.
Thus we obtain the commutative diagram
$$
\begin{CD}
R(\tau\pi i)_*(\bC \otimes \bC) @>>> R(\tau\pi i)_*\bC \\
@VVV @VVV \\
R(\tau\pi i)_*
((\pi i)^{-1}\omega^{\cdot,\log}_X \otimes (\pi i)^{-1}\omega^{\cdot,\log}_X)
@>>>
R(\tau\pi i)_*((\pi i)^{-1}\omega^{\cdot,\log}_X) \\
@AAA @AAA \\
\omega^{\cdot}_X[u] \otimes \omega^{\cdot}_X[u] @>>> \omega^{\cdot}_X[u] \\
@VVV @VVV \\
\omega^{\cdot}_{X/0} \otimes \omega^{\cdot}_{X/0} @>>> \omega^{\cdot}_{X/0}
\end{CD}
$$
on $X$.
On the other hand,
the diagram
$$
\begin{CD}
\omega^{\cdot}_{X/0} \otimes \omega^{\cdot}_{X/0} @>>> \omega^{\cdot}_{X/0} \\
@VVV @VVV \\
\cC(\omega^{\cdot}_{X_{\bullet}/0}) \otimes \cC(\omega^{\cdot}_{X_{\bullet}/0})
@>{\Phi_{/*}}>> \cC(\omega^{\cdot}_{X_{\bullet}/0}) \\
@AAA @AAA \\
K_{\bC} \otimes K_{\bC} @>>{\Phi_{\bC}}> K_{\bC}
\end{CD}
$$
is commutative as in 6.5 and Lemma 6.6 of \cite{F},
where $\Phi_{/*}$ and $\Phi_{\bC}$
are the morphisms defined in Definition 6.4 of \cite{F}.
Hence we obtain the conclusion.
\end{proof}

The next proves the compatibility (3) in \ref{para: three compatibilities}. 

\begin{prop}
The stalk of the trace map
$\bC \otimes H^{2n}_{\bQ,z} \to \bC$
is compatible with $\tr: H^{2n}(X, K_{\bC}) \to \bC$
under the identification \eqref{eq:identification1}.
\end{prop}
\begin{proof}
Let $X=\bigcup_{\lambda \in \Lambda}X_{\lambda}$
be the irreducible decomposition of $X$.
Since $X$ is compact,
$\Lambda$ is a finite set.
A dense open subset $U$ of $X$ is defined by
$$
U=X \setminus \bigcup_{\lambda \not= \mu}(X_{\lambda} \cap X_{\mu}),
$$
which is a complex manifold of dimension $n$
because all the irreducible components $X_{\lambda}$ are smooth.
By setting $U_{\lambda}=U \cap X_{\lambda}$ for all $\lambda \in \Lambda$,
$U$ is a disjoint union of all $U_{\lambda}$,
that is, $U=\coprod_{\lambda \in \Lambda}U_{\lambda}$.
Because the morphism $\tau\pi i: X^{\log}_z \longrightarrow X$
induces an isomorphism $(\tau\pi i)^{-1}(U) \longrightarrow U$,
we identify $(\tau\pi i)^{-1}(U)$ and $U$.
Thus $U$ is regarded as a dense open subset of $X^{\log}_z$
and the inclusion is denoted
by $j: U \hookrightarrow X^{\log}_z$.
The morphism of zero extension
$H^q_c(U, \bC) \to H^q(X, \bC)$
is obtained from the canonical inclusion
$j_!\bC \to \bC$ on $X^{\log}_z$.
According to Proposition 5.8 and Theorem 5.10 of \cite{NO},
the germ of the trace map
$R^{2n}f^{\log}_*\bC \to \bC$
is defined as the unique morphism
$H^{2n}(X^{\log}_z, \bC) \to \bC$
such that the composite
$$
H^{2n}_c(U, \bC) \to H^{2n}(X^{\log}_z, \bC) \to \bC
$$
coincides with the usual trace map on $U$,
the morphism given by the integration on $U$.
Therefore it suffices to prove that
the composite
$$
H^{2n}_c(U, \bC)
\to H^{2n}(X^{\log}_z, \bC)
\overset{\cong}{\longrightarrow} H^{2n}(X, K_{\bC})
\overset{\tr}{\longrightarrow} \bC
$$
is given by the integration on $U$,
where the middle isomorphism is induced
from the identification \eqref{eq:identification1}.
We have the commutative diagram
$$
\begin{CD}
j_!\bC @>>> j_!j^{-1}(\pi i)^{-1}\omega^{\cdot,\log}_X
       @<<< j_!j^{-1}(\tau\pi i)^{-1}\omega^{\cdot}_X[u] \\
@VVV @VVV @VVV \\
\bC @>>> (\pi i)^{-1}\omega^{\cdot,\log}_X
    @<<< (\tau\pi i)^{-1}\omega^{\cdot}_X[u],
\end{CD}
$$
where the morphisms in the bottom
are the ones given in \ref{para: identifications}
and the morphisms in top
is induced from the bottom.
Then the commutative diagram
$$
\begin{CD}
H^{2n}_c(U, \bC)
@>>> H^{2n}_c(U, (\pi i)^{-1}\omega^{\cdot,\log}_X)
@<<< H^{2n}_c(U, \omega^{\cdot}_X[u]) \\
@VVV @VVV @VVV \\
H^{2n}(X^{\log}_z, \bC)
@>>> H^{2n}(X^{\log}_z, (\pi i)^{-1}\omega^{\cdot,\log}_X)
@<<< H^{2n}(X, \omega^{\cdot}_X[u])
\end{CD}
$$
is obtained by $R(\tau\pi i)_!=R(\tau\pi i)_*$
and by the fact that $\tau \pi i j: U \to X$
is the inclusion from the open subset $U$ to $X$.
Note that $\tau\pi i: X^{\log}_z \to X$ is a proper continuous map.
Combining with the commutative diagrams
$$
\begin{CD}
H^{2n}_c(U, \bC) @= H^{2n}_c(U, \bC) \\
@VVV @VVV \\
H^{2n}_c(U, (\pi i)^{-1}\omega^{\cdot,\log}_X)
@<<< H^{2n}_c(U, \omega^{\cdot}_X[u]),
\end{CD}
\qquad
\begin{CD}
H^{2n}_c(U, \omega^{\cdot}_X[u])
@>>> H^{2n}_c(U, K_{\bC}) \\
@VVV @VVV \\
H^{2n}(X, \omega^{\cdot}_X[u])
@>>> H^{2n}(X, K_{\bC}),
\end{CD}
$$
the problem is reduced to prove that
the composite
$$
H^{2n}_c(U, \bC)
\to H^{2n}_c(U, K_{\bC})
\to H^{2n}(X, K_{\bC})
\overset{\tr}{\longrightarrow} \bC
$$
is given by the integration on $U$,
where the first arrow is induced from the composite
$\bC \to \omega^{\cdot}_X[u] \to K_{\bC}$.
It is trivial that this morphism $\bC \to K_{\bC}$
factors as $\bC \to W_0K_{\bC} \to K_{\bC}$.
Moreover the diagram
$$
\begin{CD}
H^{2n}(X, W_0K_{\bC}) @>>> H^{2n}(X, K_{\bC}) \\
@VVV @VV{\tr}V \\
H^{2n}(X, \gr^W_0K_{\bC}) @>>{\Theta_{\bC}}> \bC
\end{CD}
$$
is commutative by the definition of the trace morphism
\cite[Definition 7.11]{F},
where $\Theta_{\bC}$ is the morphism
defined in Definition 7.7 of \cite{F}.
Thus it suffices to compute the composite
$$
H^{2n}_c(U, \bC)
\to H^{2n}_c(U, W_0K_{\bC})
\to H^{2n}_c(U, \gr^W_0K_{\bC})
\to H^{2n}(X, \gr^W_0K_{\bC})
\overset{\Theta_{\bC}}{\longrightarrow} \bC.
$$
This morphism coincides with the composite
\begin{equation}
\label{eq:trace map 1}
H^{2n}_c(U, \bC)
\to H^{2n}_c(U, \gr^{\delta W}_0\cC(\omega^{\cdot}_{X_{\bullet}}))
\to H^{2n}(X, \gr^{\delta W}_0\cC(\omega^{\cdot}_{X_{\bullet}}))
\overset{\Theta_{\bC,0}}{\longrightarrow} \bC,
\end{equation}
where $\Theta_{\bC,0}$ is the morphism
defined in Definition 7.7 of \cite{F},
from the fact that $\bC \to W_0K_{\bC}$ factors through the subcomplex
$(\delta W)_0\cC(\omega^{\cdot}_{X_{\bullet}})$
and from the commutativity of the diagram
$$
\begin{CD}
H^{2n}(X, \gr^{\delta W}_0\cC(\omega^{\cdot}_{X_{\bullet}}))
@>{\Theta_{\bC,0}}>>
\bC \\
@VVV @| \\
H^{2n}(X, \gr^W_0K_{\bC})
@>>{\Theta_{\bC}}> \bC
\end{CD}
$$
by 7.8 of \cite{F}.
Furthermore, we have the commutative diagram
$$
\begin{CD}
H^{2n}(X, \gr^{\delta W}_0\cC(\omega^{\cdot}_{X_{\bullet}}))
@>{\Theta_{\bC,0}}>> \bC \\
@VVV @| \\
H^{2n+1}(X, \gr^{\delta W}_1\cC(\omega^{\cdot}_{X_{\bullet}}))
@>>> \bC,
\end{CD}
$$
where the left vertical morphism
is induced from the morphism
$$
\gr^{\delta W}_0\cC(\dlog t \wedge):
\gr^{\delta W}_0\cC(\omega^{\cdot}_{X_{\bullet}})
\to \gr^{\delta W}_1\cC(\omega^{\cdot}_{X_{\bullet}})[1]
$$
defined in 5.11 of \cite{F}.
Therefore the morphism \eqref{eq:trace map 1} coincides
with the composite
\begin{equation}
\label{eq:trace map 2}
\begin{split}
H^{2n}_c(U, \bC)
\to H^{2n}_c(U, \gr^{\delta W}_0\cC(\omega^{\cdot}_{X_{\bullet}}))
&\to H^{2n+1}_c(U, \gr^{\delta W}_1\cC(\omega^{\cdot}_{X_{\bullet}})) \\
&\to H^{2n+1}(X, \gr^{\delta W}_1\cC(\omega^{\cdot}_{X_{\bullet}}))
\to \bC.
\end{split}
\end{equation}
As in Remark 5.10 of \cite{F},
we have
\begin{align*}
&\gr^{\delta W}_0\cC(\omega^{\cdot}_{X_{\bullet}})|_U
=\bigoplus_{\lambda \in \Lambda}
\gr^W_0\omega^{\cdot}_{U_{\lambda}}
=\bigoplus_{\lambda \in \Lambda}\Omega^{\cdot}_{U_{\lambda}} \\
&\gr^{\delta W}_1\cC(\omega^{\cdot}_{X_{\bullet}})|_U
=\bigoplus_{\lambda \in \Lambda}
\gr^W_1\omega^{\cdot}_{U_{\lambda}}
\cong \bigoplus_{\lambda \in \Lambda}\Omega^{\cdot}_{U_{\lambda}}[-1]
\end{align*}
and then
\begin{align*}
H^{2n}_c(U, \gr^{\delta W}_0\cC(\omega^{\cdot}_{X_{\bullet}}))
=\bigoplus_{\lambda \in \Lambda}
H^{2n}_c(U_{\lambda}, \Omega^{\cdot}_{X_{\lambda}}) \\
H^{2n+1}_c(U, \gr^{\delta W}_1\cC(\omega^{\cdot}_{X_{\bullet}}))
\cong
\bigoplus_{\lambda \in \Lambda}
H^{2n}_c(U_{\lambda}, \Omega^{\cdot}_{X_{\lambda}}),
\end{align*}
under which the morphism
$H^{2n}_c(U, \gr^{\delta W}_0\cC(\omega^{\cdot}_{X_{\bullet}}))
\to H^{2n+1}_c(U, \gr^{\delta W}_1\cC(\omega^{\cdot}_{X_{\bullet}}))$
in \eqref{eq:trace map 2}
is identified with the identity.

By $U=\coprod_{\lambda \in \Lambda}U_{\lambda}$
and by Definition 7.7 of \cite{F} again,
it is easy to check that the composite
$$
\bigoplus_{\lambda \in \Lambda}
H^{2n}_c(U_{\lambda}, \Omega^{\cdot}_{X_{\lambda}})
\cong
H^{2n+1}_c(U, \gr^{\delta W}_1\cC(\omega^{\cdot}_{X_{\bullet}}))
\to H^{2n+1}(X, \gr^{\delta W}_1\cC(\omega^{\cdot}_{X_{\bullet}}))
\to \bC
$$
is given by
$$
\sum_{\lambda \in \Lambda}\int_{U_{\lambda}}=\int_U
$$
as desired.
\end{proof}

\section{Degeneration of log Hodge and log de Rham spectral sequences}
\label{s:dege}
  In this section, we prove Theorems \ref{t:Lef}, \ref{t:HdRss}, and \ref{t:variant}. 

\begin{para}
\label{potentialsss}
  First we prove Theorem \ref{t:Lef}, the nonk\'et version of 
Theorem \ref{t:HdRss} (i.e., Theorem \ref{t:variant} (1), (2)), 
and, for a later use in Section \ref{s:pf}, the nonk\'et version of Theorem \ref{t:main} (i.e., Corollary \ref{c:main}) 
for a family $X$ over the standard log point $S=0$ such that there is a log blow-up $$p \colon X' \to X,$$ 
where $X'$ is a projective strict log deformation.

\noindent
  (We remark that in this case, $X \to S$ is saturated and, in particular, it satisfies the assumption of Corollary \ref{c:main}.
  We prove it in Proposition \ref{p:complement} below, though it is not logically necessary.)
  These are reduced to the cases of strict log deformation, which is already proved in Section \ref{s:sss}, by the log blow-up invariance of log Betti cohomology 
(\cite{KN:lc} {\sc Proposition} 5.3) and by the natural isomorphism 
$$R p_* \omega^q_{X'/0} = \omega^q_{X/0}$$ for any $q$.
  Since $\omega^q_{X/0}$ is locally free and $\omega^q_{X'/0}$ is the pullback of it, the last isomorphism is reduced to the following proposition. 
\end{para}

\begin{prop}
  Let $X \to 0$ be a log deformation and let $p \colon X' \to X$ be a log blow-up. 
  Then we have the natural isomorphism $$\cO_X \overset \cong \to R p_* \cO_{X'}.$$ 
\end{prop}

\begin{pf}
  Since the statement is local on $X$, we may assume that there is a log blow-up $q\colon W' \to W$ of toric varieties endowed with a morphism from $W$ to 
the log affine line whose special fiber is isomorphic to $p$. 
  Then we have the exact sequences 
\begin{align*}
&0 \to \cO_W \to \cO_W \to i_*\cO_X \to 0 \text{\ and} \\
&0 \to \cO_{W'} \to \cO_{W'} \to i'_*\cO_{X'} \to 0, 
\end{align*}
where $i\colon X \to W$ and $i'\colon X' \to W'$ are the canonical closed immersions. 
  From these exact sequences and the natural isomorphism 
$$R q_* \cO_{W'}=\cO_W$$
(see \cite{KKMS} p.44, Chapter I, Section 3, Corollary 1 to Theorem 12),
we have $Rq_*i'_*\cO_{X'} = i_*\cO_X$, which implies the desired isomorphism. 
\end{pf}

  Thus we have proved Theorem \ref{t:Lef}, Theorem \ref{t:variant} (1), (2), and Corollary \ref{c:main} 
for $X$ as in \ref{potentialsss}.

\begin{para}
\label{makesat}
  We prove Theorem \ref{t:HdRss}.
  By the arguments in \cite{IKN} Section 7 and \cite{IKNe}, we reduce to the case where the base is an fs log point as follows. 
  (In fact, the argument here gives a simplification of the one in \cite{IKN}, \cite{IKNe}.
  Note that the assumption in {\sc Lemma} (7.1') in \cite{IKNe} that the sheaves $R^nf_*\omega^{\cdot}_{X/Y}$ are locally free of finite type on $Y_{\mathrm{et}}$ for all $n$ is automatic by \cite{IKN} {\sc Lemma} (A.4.1) and \cite{IKN} {\sc Theorem} (6.3) (2) (cf.\ the beginning of \ref{pf-l:HD}).)

  Since the question is k\'et local on $S$ and $f$ is projective and exact, by \cite{IKN} {\sc Proposition} (A.4.3), we may assume that $f$ is saturated. 
  Hence, Theorem \ref{t:HdRss} is reduced to the following lemma, which 
is an analytic variant of \cite{IKNe} {\sc Lemma} (7.1').
  (The assumption is, however, not identical.  Cf.\ the beginning of \ref{pf-l:HD}.)
\end{para}

\begin{lem}
\label{l:HD}
  Let $f \colon X \to S$ be a projective, vertical, saturated, and log smooth morphism of fs log analytic spaces. 
  Then the following hold. 

$(1)$ The spectral sequence 
$$(*) \qquad  E_1^{pq} = R^qf_{*}\omega^{p}_{X/S} \Rightarrow 
R^{p+q}f_{*}\omega^{\cdot}_{X/S}$$
degenerates from $E_1$ and has locally free initial terms. 

$(2)$ The spectral sequence 
$$(**) \qquad 
E_1^{pq} = R^qf_{\ket*}\omega^{p}_{X/S,\ket} \Rightarrow 
R^{p+q}f_{\ket*}\omega^{\cdot}_{X/S,\ket}$$
degenerates from $E_1$ and has locally free initial terms. 

$(3)$ For any interval $[a,b]$ and any integer $n$, the natural map 
$$\varepsilon^* R^n f_*\omega_{X/S}^{[a,b]} \to R^n f_*\omega_{X/S,\ket}^{[a,b]}$$
is an isomorphism.
  In other words, the spectral sequence $(**)$ is the inverse image by $\ep$ of the spectral sequence $(*)$. 
\end{lem}

  Note that the above (1) and (3) give a proof of Theorem \ref{t:variant}.

\begin{para}
\label{pf-l:HD}
  We prove Lemma \ref{l:HD}. 
  First we prove that the sheaves $R^nf_*\omega^{\cdot}_{X/S}$ are locally free of finite type on $S$ for all $n$. 
  Since $f$ is saturated, by \cite{IKN} {\sc Lemma} (A.4.1), 
the cokernel of $\overline M^{\gp}_{S,f(x)} \to \overline M^{\gp}_{X,x}$ is torsion-free for any $x \in X$. 
  Then, by \cite{IKN} {\sc Theorem} (6.3) (2), $R^nf_*\omega^{\cdot}_{X/S}$ are locally free of finite type on $S$ for all $n$. 

  The reduction of (2) and (3) to (1) is similar to the argument in \cite{IKNe} with the use of the fact that the adjunction homomorphism $\cO_S \to R\ep_*\cO_{S}^{\ket}$ is an isomorphism, 
which is by \cite{IKN} {\sc Proposition} (3.7) (5) (cf.\ \cite{IKN} (7.1.1)). 
  Further, we reduce (1) to the case where $S$ is an fs log point by the usual argument by the lengths in \cite{D1} (cf.\ \cite{FN2} Appendix). 
  The rest is to prove (1) in the case where $S$ is an fs log point. 
  Here we use the following variant of the semistable reduction theorem of D.\ Mumford. 
\end{para}

\begin{prop}
\label{p:ssrt}
  Let $s = (\Spec \bC, \bN\oplus \bC^{\times})$ be the standard log point over $\bC$. 
  Let $X \to s$ be a projective vertical log smooth morphism of fs log analytic spaces. 
  Then there are a positive  integer $n$ and a log blow-up $X' \to X \times_ss_n$, where 
$s_n:=(\Spec \bC, \frac 1n \bN\oplus \bC^{\times})$, such that the composition $X' \to s_n$ is a projective strict log deformation. 
\end{prop}

  This is a corollary to \cite{V} Proposition 2.4.2.1 
(cf.\ \cite{KKN} Remark after {\sc Assumption} 8.1). 
  See also \cite{Saito} Theorem 1.8 and Theorem 2.9. 

\begin{para}
\label{reduce2sss}
  We prove the case of (1) of Lemma \ref{l:HD} where $S=(\Spec \bC, P \oplus \bC^{\times})$ with $P$ being an fs monoid. 
  Take a local homomorphism $P \to \bN$, and we apply Proposition \ref{p:ssrt} to the base-changed family on the standard log point obtained by $P\to \bN$.
  Then we have a positive integer $n$ such that the base-changed family obtained by $P \to \bN \overset n \to \bN$ satisfies the condition in \ref{potentialsss}. 
  Hence, by \ref{potentialsss}, we see that (1) of Lemma \ref{l:HD} is valid for the base-changed family. 
  But, since the original family $f$ is saturated and the underlying morphism of fs log points is an isomorphism, the underlying morphism of $f$ and that of the base-changed family are isomorphic to each other via the natural isomorphisms. 
  Further, the log de Rham complex for the base-changed family is the pullback of that for the original family.
  Hence the spectral sequence concerned for $f$ is isomorphic to that for the base-changed one so that we see that (1) of Lemma \ref{l:HD} is valid for $f$, 
which completes the proof of Lemma \ref{l:HD}. 
  Therefore, Theorem \ref{t:HdRss} and Theorem \ref{t:variant} are also proved. 
\end{para}

\begin{para}
  We prove Theorem \ref{t:Lef} in the same way.  
  Let $s \in S$ and it is enough to check the statement at $s$. 
  By the argument in \ref{reduce2sss}, there is a morphism $S' \to S$ from the standard log point 
whose image is $\{s\}$ such that the base-changed object satisfies the condition in \ref{potentialsss}. 
  Then by the log proper base change theorem in the log Betti cohomology (\cite{KN:lc} {\sc Proposition} 5.1 (cf.\ \cite{KN:lc} {\sc Remark} 5.1.1); here we use the assumption that our family is exact), we reduce the statement at $s$ to the special case in \ref{potentialsss}. 
\end{para}

As a corollary of Theorem \ref{t:Lef}, we have the following hard Lefschetz statement. 

\begin{prop}
\label{p:HL}
  The homomorphism 
$$l^{n-q} \colon H^q_{\bQ} \to H^{2n-q}_{\bQ} \qquad(q \le n)$$ 
is an isomorphism. 
\end{prop}

  Here we prove the fact mentioned in \ref{potentialsss}. 

\begin{prop}
\label{p:complement}
  Let $S$ be the standard log point and $f\colon X \to S$ a morphism of fs log analytic spaces.
  Let $X' \to X$ be a log blow-up. 
  If $X' \to X \to S$ is saturated, then $f$ is saturated. 
\end{prop}

\begin{pf}
  By \cite{K} {\sc Corollary} (4.4) (ii), $X$ is integral over $S$.
  Let $x \in X$. 
  Let $Q=\overline M_{X,x}$ and 
$P=\overline M_{S,f(x)}\cong \bN$. 
  To prove that $h\colon P \to Q$ is saturated, we use the criterion \cite{T} Theorem I.5.1 for saturated morphisms.   
  Let $\frak q$ be a prime ideal of $Q$ of height 1 lying over the maximal ideal of $P$ (i.e., the inverse image in $P$ of $\frak q$ is the maximal ideal). 
  Take a point $x'$ of $X'$ lying over $x$ such that there is a prime ideal $\frak q'$ of $Q':=\overline M_{X',x'}$ lying over $\frak q$. 
  (Note that, then the height of $\frak q'$ is 1.) 
  By assumption, the ramification index of $X' \to X \to S$ at $\frak q'$ is 1. 
  This implies that the ramification index of $f$ at $\frak q$ is 1. 
  By \cite{T} Theorem I.5.1, we conclude that $h$ is saturated, as desired. 
\end{pf}

\section{Proof of Theorem \ref{t:main}}\label{s:pf}
  We prove Theorem \ref{t:main}. 

\begin{para}
\label{pf1}
  Since the statement is k\'et local on the base $S$, by the argument in \ref{makesat}, we may assume that $f$ is saturated and hence the conclusion in Lemma \ref{l:HD} 
holds for $f$. 
  In particular, we may assume that the triple 
$\bigl(H^q_{\bZ}, (R^qf_{\ket*}\omega^{\cdot\ge p}_{X/S,\ket})_{p \in \bZ}, (\ , \ )\bigr)$
concerned is the pullback of the nonk\'et triple 
$$\bigl(H^q_{\bZ}, (R^qf_{*}\omega^{\cdot\ge p}_{X/S})_{p \in \bZ}, (\ , \ )\bigr).$$
  Hence it is enough to show that this nonk\'et triple makes a nonk\'et PLH.
\end{para}

\begin{para}
  We prove that the nonk\'et triple in \ref{pf1} satisfies the 
condition (1) in \ref{plh}. 
  To see this, by the definition of the pairing in Section \ref{s:mt} and by the fact that $l$ in \ref{pairing} is a homomorphism of $(1,1)$-type, 
it is enough to show that for any section of $\tau^*F^p$ $(p \in \bZ)$, its Lefschetz components belong to $\tau^*F^p$. 
  It suffices to show that the homomorphism 
$$l^{n-q} \colon H^q \to H^{2n-q} \qquad (q \le n)$$ 
of nonk\'et triples, which is underlain by the isomorphism of $\bQ$-vector spaces by Proposition \ref{p:HL}, is strictly compatible with the filtration $(\tau^*F^p)_{p\in\bZ}$. 
  This is reduced to the strict compatibility with the filtration $(F^p)_{p\in\bZ}$, and is further reduced by 
Nakayama's lemma to the case where $S=\{s\}$ is an fs log point. 
  If the base has the trivial log structure, the last strict compatibility is classical. 
  Hence we may assume that $P=\overline M_{S,s}$ is isomorphic to $\bN$. 
  Then, by the argument in \ref{reduce2sss}, \ref{potentialsss} implies 
that there is a positive integer $n$ such that the base-changed nonk\'et triples on the standard log point 
obtained by $P \cong \bN \overset n \to \bN$ are nonk\'et PLHs. 
  Since $f$ is saturated, we may assume that 
both $H^q$ and $H^{2n-q}$ are nonk\'et PLHs. 
  Then, by the condition (2) in \ref{plh}, the desired strict compatibility is reduced to the usual nonlog statement that any homomorphism of Hodge structures is strictly compatible with the Hodge filtration. 
  Alternatively, we can apply the next general proposition. 
\end{para}

\begin{prop}
  Let $S$ be an fs log analytic space. 
  Then any homomorphism $f \colon H_1 \to H_2$ of log Hodge structures on $S$ is strictly compatible with 
the Hodge filtrations by $\cO_S$-modules and with that by $\cO_S^{\log}$-modules. 
\end{prop}

\begin{pf}
  Consider the complex 
$$ \cdots \to 0 \to H_1 \overset f \to H_2 \to 0 \to \cdots$$ 
of $\cO_S$-modules (resp.\ $\cO_S^{\log}$-modules).
  The strict compatibility with the Hodge filtration by $\cO_S$-modules (resp.\ $\cO_S^{\log}$-modules) 
is equivalent to the $E_1$-degeneracy of the associated spectral sequence of this complex. 
  Since $\cO^{\log}_S$ is flat over $\cO_S$, 
the statement for $\cO_S^{\log}$-modules is reduced to that for $\cO_S$-modules. 
  By the usual argument by the lengths in \cite{D1} (cf.\ \cite{FN2} Appendix), 
we reduce to the case where $S$ is an fs log point. 
  (Note that the construction of the spectral sequence is compatible with the base change.)
  By the definition of log Hodge structure, the last one is reduced to the usual nonlog statement that any homomorphism of Hodge structures is strictly compatible with the Hodge filtration. 
\end{pf}

\begin{para}
  Thus the condition (1) in \ref{plh} is satisfied.
  To complete the proof of Theorem \ref{t:main}, the rest is to verify the conditions (2) and (3) in \ref{plh}. 
  By Theorem \ref{t:variant} (2), we may assume that the base $S=\{s\}$ is an fs log point. 
  If $S$ has the trivial log structure, the conclusion is classical. 
  Hence we may assume that $P=\overline M_{S,s}$ is isomorphic to $\bN$. 
  Then, again by the argument in \ref{reduce2sss}, \ref{potentialsss} implies 
that there is a positive integer $n$ such that the base-changed nonk\'et triple on the standard log point 
obtained by $P \cong \bN \overset n \to \bN$ is a PLH. 
  Hence, by the next Lemma \ref{l:NtoP}, we conclude that the original nonk\'et triple satisfies the conditions (2) and (3) in \ref{plh}, which completes the proof of Theorem \ref{t:main}.
\end{para}

\begin{lem}
\label{l:NtoP}
  Let $H=(L_{\bZ}, (F^p)_{p\in \bZ}, (\ ,\ ))$ be a nonk\'et triple as in $\ref{triple}$ on the standard log point.
  Let $n$ be a positive integer. 

$(1)$ 
  If the base-changed nonk\'et triple on the standard log point obtained by 
$\bN \overset n \to \bN$ satisfies the positivity ($\ref{plh}$ $(2)$), then $H$ also satisfies the positivity.

$(2)$ 
  If the base-changed nonk\'et triple on the standard log point obtained by 
$\bN \overset n \to \bN$ satisfies the Griffiths transversality ($\ref{plh}$ $(3)$), then $H$ also satisfies the Griffiths transversality. 
\end{lem}

\begin{pf}
  (1) Let $S' \to S$ be the induced morphism of the standard log points, let $t' \in S^{\prime\log}$, and let $t \in S^{\log}$ be the image of $t$. 
  Let $g'$ be the generator of $\bN$ on $S'$ and let $h'\colon \cO^{\log}_{S', t'} \to \bC$ be a $\bC$-algebra homomorphism. 
  By the assumption, there is a positive real number $\ep$ such that, if $\exp(h'(\log(g')))< \ep$, then the specialization with respect to the composite 
$\cO^{\log}_{S,t} \to \cO^{\log}_{S', t'} \overset {h'} \to \bC$ is a polarized Hodge structure. 
  Let $g$ be the generator of $\bN$ on $S$ and let $h \colon \cO^{\log}_{S,t} \to \bC$ be a $\bC$-algebra homomorphism. 
  Then $h$ factors uniquely through an $h'\colon \cO^{\log}_{S', t'} \to \bC$ the image of $\log(g')$ by which is $h(\log(g))/n$. 
  Assuming $\exp(h(\log(g)))< \ep^n$, we have $\exp(h'(\log(g')))< \ep$, and the specialization with respect to $h$ is a polarized Hodge structure. 

  (2) We use the same notation as in the proof of (1). 
  Let $(F^p)_{p \in \bZ}$ be the Hodge filtration. 
  Let 
  $a \in \omega^{1,\log}_S \otimes H_{\bZ}$ belong to the image of $\cO^{\log}_S \otimes F^p$ over $S$. 
  We want to prove $a \in \omega^{1,\log}_S \otimes F^{p-1}$. 
  But there is a natural isomorphism from the pullback of $\omega^{1,\log}_S \otimes H_{\bZ}$ to 
$\omega^{1,\log}_{S'} \otimes H_{\bZ}$, which preserves the Hodge filtration. 
  By assumption, the image of $a$ in $\omega^{1,\log}_{S'} \otimes H_{\bZ}$ belongs to $\omega^{1,\log}_{S'} \otimes F^{p-1}$.
  Hence $a$ belongs to $\omega^{1,\log}_S \otimes F^{p-1}$, as desired. 
\end{pf}

  Finally we discuss a difficulty in trying to generalize Theorem \ref{t:main} to the case of general base. 

  As for Griffiths transversality, we have the following proposition. 

\begin{prop}
\label{p:GT}
  Let $H=(L_{\bZ}, (F^p)_{p\in \bZ}, (\ ,\ ))$ be a nonk\'et triple as in $\ref{triple}$ on an fs log analytic space $S$. 

  $(1)$  Assume that $S$ is an fs log point $(\Spec \bC, P \oplus \bC^{\times})$, where $P$ is an fs monoid. 
  Let $a_i\colon P \to \bN$ $(i \in I)$ be a family of local homomorphisms. 
  Assume that the set $\{a_i^{\gp}\,|\,i \in I\}$ spans $\Hom(P, \bQ)$.
  If the base-changed nonk\'et triples on the standard log point obtained by $a_i$ satisfy the Griffiths transversality, then $H$ also satisfies the Griffiths transversality. 

  $(2)$ Let $S' \to S$ be a log blow-up. 
  If the pullback of $H$ on $S'$ satisfies the Griffiths transversality, then $H$ also satisfies the Griffiths transversality. 
\end{prop}

\begin{pf}
  $(1)$  In the following, we identify an $\cO_S$-module with a $\bC$-vector space. 
  Let $(t_j)_{1\le j \le n}$ be a $\bZ$-basis of $P^{\gp}$. 
  Then $(d\log t_j)_j$ is a $\bC$-basis of $\omega_S$. 
  Let $\sum_j h_j \otimes d\log t_j$ $(h_j \in H)$ belong to the image of $F^p$ in $H \otimes \omega_S$. 
  We want to prove any $h_j \in F^{p-1}$. 
  For each $i \in I$ and $j$, let $m_{ij}=a_i^{\gp}(t_j)\in \bN^{\gp}=\bZ$. 
  Then, the image of $\sum_{j} h_j \otimes d\log t_j$ in $H \otimes \omega_{S_0}$ via $a_i$, where $S_0$ is the standard log point, 
is $\sum_j m_{ij}h_j \otimes d\log t$, where $t$ is the generator of $\bN$. 
  By assumption, $\sum_j m_{ij}h_j$ belongs to $F^{p-1}$ and we have an equality 
$$(m_{ij})(h_j)=(f_i)$$ 
of matrices, 
where $f_i$ belongs to $F^{p-1}$. 
  Again by the assumption, we may assume that $I=\{1, \ldots, n\}$ and that $(m_{ij})$ is a regular matrix, and each $h_j$ is a linear combination of $f_i$. 
  Hence, $h_j \in F^{p-1}$ as desired. 

  $(2)$ To reduce (2) to (1), it is enough to show that any morphism from the standard log point $S_0$ to $S$ factors through $S'$. 
  Consider the fiber product $S_0 \times_SS'$ of fs log analytic spaces.  
  Then the projection $S_0 \times_SS' \to S_0$ is a log blow-up of the standard log point so that it is an isomorphism.  
  Thus $S_0 \to S$ factors through $S'$. 
\end{pf}

\begin{para}
\label{difficulty}
  In Proposition \ref{p:GT} (1), the family of all local homomorphisms from $P$ to $\bN$ satisfies the assumption so that, by the argument in \ref{reduce2sss}, even for the general base, we can prove that the nonk\'et triple for a saturated $f$ satisfies the Griffiths transversality. 
  Thus, by the argument in the above of this section, we can prove the following statements.

1. Let $f\colon X \to S$ be a projective, vertical, exact, and log smooth morphism of fs log analytic spaces. 
  Then, for each $q \in \bZ$,  
$$\bigl(H^q_{\bZ}, (R^qf_{\ket*}\omega^{\cdot\ge p}_{X/S,\ket})_{p \in \bZ}, (\ , \ )\bigr)$$
is a prePLH on $S_{\ket}$ which satisfies the Griffiths transversality after pulling back to each point of $S$. 

2.  Let $f\colon X \to S$ be a projective, vertical, saturated, and log smooth morphism of fs log analytic spaces. 
  Then, for each $q \in \bZ$,  
$$\bigl(H^q_{\bZ}, (R^qf_{*}\omega^{\cdot\ge p}_{X/S})_{p \in \bZ}, (\ , \ )\bigr)$$
is a nonk\'et prePLH on $S$ which satisfies the Griffiths transversality after pulling back to each point of $S$. 

  However, as for the positivity, there is a counter example of rank 3 for the analogue of Proposition \ref{p:GT} (1) as explained below.
  Hence the argument in \ref{reduce2sss} does not work for the positivity in case of general base.
  Further, there is also a counter example for the analogue of Proposition \ref{p:GT} (2).
  Thus we cannot expect that the problem is reduced even to weakly semistable case (by log blowing up the base). 

  We describe the counter examples mentioned in the above. 
  Let the notation be as in \cite{KU} 12.2 with $n=3$. 
  Let $a$ be an element of $H'_{0,\bQ}=\bQ$ satisfying
$\langle a,a\rangle'_0=-1$. 
  Let $F'$ be the unique element of $\check D'$. 
  Then, by \cite{KU} {\sc Lemma} 12.2.6 (ii), for $y\in \bR$, $\exp(iyN_a)s(F') \in D$ if and only if $y \not=0$. 

  Now let $N_1=-N_2=N_a$, and let $F=\exp(iN_a)s(F')$. 
  Then $(N_1,N_2,F)$ generates a triple as in \ref{triple} over the fs log point $S:=(\Spec \bC, \bN^2\oplus \bC^{\times})$.
  This triple is not a PLH. 
  In fact, let $y_1, y_2 \in \bR$. 
  Then $\exp(iy_1N_1+iy_2N_2)F =\exp(i(y_1-y_2+1)N_a)s(F')\in D$ if and only if $y_1-y_2+1\not=0$, 
but for any $K \in \bR$, there exist $y_1, y_2 >K$ such that $y_1-y_2+1=0$. 
  On the other hand, for any morphism from the standard log point to $S$, the pullback of this triple is a PLH. 
  In fact, let $m_1, m_2$ be positive integers. 
  Then the pullback of the triple via the morphism induced by $\bN^2 \to \bN; e_j \to m_j$, where $(e_j)_j$ is the canonical basis of $\bN^2$, is a PLH because ${m_1}y-{m_2}y+1\not=0$ for any sufficiently large $y$. 
  (Note that the orthogonality and the Griffiths transversality are always satisfied (cf.\ \cite{KU} 12.2.2 and \cite{KU} {\sc Lemma} 12.2.6 (iii)).)
  Thus this gives a counter example for the analogue of Proposition \ref{p:GT} (1) with respect to the positivity. 

  Further, let $S'$ be the log blow-up of $S$ by the log ideal generated by $e_1, e_2$. 
  Then the pullback of the triple over $S'$ is a PLH because 
\begin{align*}
&(y'_1+y'_2)-y'_2+1=y'_1+1 \not=0 \ \text{ for sufficiently large }y'_1, y'_2, \text { and} \\ 
&y'_1-(y'_1+y'_2)+1=-y'_2+1\not=0 \ \text{ for sufficiently large }y'_1, y'_2.
\end{align*}
  Hence, this gives a counter example for the analogue of Proposition \ref{p:GT} (2) with respect to the positivity. 
\end{para}

\section{Log Picard varieties and log Albanese varieties}\label{s:alb}
  As an application, we construct log Picard variety and log Albanese variety for an $f$ as in Corollary \ref{c:main}. 

\begin{thm}
\label{t:good}
  Let $f\colon X \to S$ be a projective, vertical, saturated and log smooth morphism of fs log analytic spaces. 
  Assume that the log rank of $S$ is $\leq 1$ (cf.\ Theorem $\ref{t:main}$). 
  Then $f$ is good both in the sense of {\rm \cite{KKN}} $7.1$ and in the sense of {\rm \cite{KKN}} $7.2$. 
  The quasi-isomorphisms in the condition $\rm{(i)}$ in {\rm \cite{KKN}} $7.1$ and in the condition $\rm{(i)}$ in {\rm \cite{KKN}} $7.2$ coincide. 
\end{thm}

\begin{sbrem}
  There is a mistake in \cite{KKN} 7.1. 
  See \cite{FN2} 2.4. 
  In the above and the below, we mean the corrected definition of the goodness.
\end{sbrem}

\begin{pf}
  As is shown in \cite{FN2} 9.2, for any proper, separated, saturated, and log smooth morphism of fs log analytic spaces, the condition (i) in \cite{KKN} 7.1 and the former part of the condition (iii) in \cite{KKN} 7.1 are always satisfied (cf.\ Proposition \ref{sp:sat}). 
  Further, \cite{FN2} 9.3 says that for such a morphism, the condition (i) in \cite{KKN} 7.2 is also satisfied and the 
quasi-isomorphisms in the condition (i) in \cite{KKN} 7.1 and in the condition (i) in \cite{KKN} 7.2 coincide. 
  Next, by Theorem \ref{t:variant} (1) and (2), the condition (ii) in \cite{KKN} 7.1 is satisfied. 
  Finally, by Corollary \ref{c:main}, the latter half of the condition (iii) in \cite{KKN} 7.1 is satisfied. 
\end{pf}

\begin{para}
  Thus we can apply the general theory in Sections 8 and 10 in \cite{KKN} in the same way as in Section 9 in \cite{FN2}. 
  We describe some of them. 
  Let $f$ be as in Theorem \ref{t:good}.
  Then, by Theorem \ref{t:good}, the log Picard variety $A^*_{X/S}$ and the log Albanese variety $A_{X/S}$ are well-defined. 
  See \cite{KKN} 7.3 for their definitions. 
  Further, {\sc Assumption} 8.1 in \cite{KKN} is satisfied by Remark after it. 
  Hence, the conclusions of all results in Sections 8 and 10 in \cite{KKN} are valid for such an $f$.  
  In particular, we have the following. 
\end{para}

  Let the notation be as in \cite{KKN} or \cite{FN2} Section 9. 

\begin{thm}
  Let $X \to S$ be as in Theorem $\ref{t:good}$. 

\smallskip

$\mathrm{(i)}$ Let 
$$\cH^1(X, {\bG}_m)_0=\Ker(\cH^1(X, {\bG}_m) \to \cH^2(X,
{\bZ})),$$
$$\cH^1(X, {\bG}_{m,\log})_0=\Ker(\cH^1(X, {\bG}_{m,\log})\to\tau_*\cH^2(X^{\log}, {\bZ})).$$

Then we have canonical embeddings
$$\cH^1(X, {\bG}_m)_0 \subset A^*_{X/S} \subset \cH^1(X, {\bG}_{m,
\log})_0.$$

\smallskip

$\mathrm{(ii)}$ We have a canonical
isomorphism
$$\cH^1(X, {\bG}_{m,\log})/\cH^1(X, {\bG}_{m,\log})_0 
\simeq \cHom({\bZ}, \cH^2(X)(1)).$$
Here $\cH^2(X)$ is the log Hodge structure on the big site constructed by Corollary $\ref{c:main}$ with $q=2$, 
and $\cHom$ is the $\cHom$ for log Hodge
structures.
\end{thm}

\begin{cor}
We have an exact sequence 
$$0
\to
\cExt^1(A_{X/S}, {\bG}_{m,\log})
\to
\cH^1(X, {\bG}_{m,\log})
\to
\tau_*\cH^2(X^{\log}, {\bZ}) \to \cH^2(X, \cO_X).$$ 
\end{cor}

\begin{prop}
  Let $X \to S$ be as in Theorem $\ref{t:good}$.
  Assume that all the fibers are connected. 

Then, for a log
complex torus
$B$ over
$S$, the following sequence is exact.
$$0 \to B(S) \to B(X) \to  \Hom(A_{X/S}, B).$$
\end{prop}

This is proved by the argument in the proof of \cite{KKN} {\sc Proposition} 10.5.  
  A problem is to prove the last arrow is surjective if $X \to S$ has a section.

\noindent Taro Fujisawa

\noindent Tokyo Denki University \\
5 Senju Asahi, Adachi, Tokyo 120-8551 \\ Japan

\noindent fujisawa@mail.dendai.ac.jp

\bigskip

\noindent Chikara Nakayama

\noindent Department of Economics \\ Hitotsubashi University \\
2-1 Naka, Kunitachi, Tokyo 186-8601 \\ Japan

\noindent c.nakayama@r.hit-u.ac.jp
\end{document}